\documentclass[a4paper,12pt]{article}
\usepackage{amsmath}
\usepackage{amsthm}
\usepackage{amssymb}
\usepackage[margin=20truemm]{geometry}
\usepackage{graphicx}
\usepackage[]{hyperref}
\usepackage{url}
\usepackage{array}
\usepackage{multirow}
\usepackage{booktabs}
\usepackage{enumitem}

\newcommand{\R}{\mathbb{R}}
\newcommand{\Z}{\mathbb{Z}}
\newcommand{\N}{\mathbb{N}}
\newcommand{\eps}{\varepsilon}
\newcommand{\supp}{\operatorname{supp}}
\newcommand{\card}{\operatorname{card}}

\theoremstyle{plain}
\newtheorem{thm}{Theorem}[section]
\newtheorem{lemma}[thm]{Lemma}
\newtheorem{prop}[thm]{Proposition}
\newtheorem{cor}[thm]{Corollary}
\theoremstyle{definition}
\newtheorem{definition}[thm]{Definition}
\theoremstyle{remark}
\newtheorem{remark}[thm]{Remark}

\title{Threshold-crossing estimates for time-weighted discrete semilinear wave equations}
\author{
Koji Wada
\footnote{
Department of Mathematics, Hokkaido University,
Kita 10, Nishi 8, Kita-ku, Sapporo, Hokkaido 060-0810, Japan.
Email address: wada.koji.n4@elms.hokudai.ac.jp}
\and
Kyouhei Wakasa
\footnote{
Faculty of Science and Engineering, Muroran Institute of Technology,
Muroran 050-8585, Japan.
Email address: wakasa@muroran-it.ac.jp}
}
\date{}
\begin{document}
\maketitle

\begin{description}[
  style=multiline,
  leftmargin=2.5cm,
  labelwidth=2.2cm
]
  \item[\textbf{Keywords:}]
  semilinear wave equation, discrete wave equation,
  threshold-crossing time, lifespan estimate,
  discrete Kato-type iteration

  \item[\textbf{MSC2020:}]
  35B44, 35L71, 39A12, 39A14, 39A22
\end{description}

\begin{abstract}
In this paper, we consider the initial value problem for a discrete semilinear wave equation with a time weight, obtained by the standard centered finite-difference scheme.
The solution is defined at every finite time step. We therefore introduce the threshold-crossing time $N_\Lambda(u^\eps)$ as the first time step at which the amplitude of the solution reaches a given level $\Lambda>0$.
The initial data are assumed to have compact support. We also assume that the sum of the initial displacement is nonnegative, and that the sum of the initial velocity is positive.
Our aim is to derive upper bounds of $N_\Lambda(u_\eps)$ for these initial data.
More precisely, for arbitrarily fixed $\Lambda > 0$, we obtain the following estimates, provided that the parameter $\eps > 0$ giving the size of the initial data is sufficiently small:
\[
N_\Lambda(u^\eps)
\le
C\eps^{-\frac{p-1}{\kappa}}
\]
in the subcritical case $\kappa>0$, and
\[
N_\Lambda(u^\eps)
\le
\exp\!\left(
C\eps^{-(p-1)}
\right)
\]
in the critical case $\kappa=0$.
For a fixed threshold, these upper bounds have the same order in the size of the initial data as the Kato-type upper bounds of the lifespan for the corresponding continuous problems.
The estimates are proved by an iteration argument: a discrete Kato-type iteration for the subcritical case and a logarithmic iteration for the critical case.
\end{abstract}

\tableofcontents

\section{Introduction}
\label{sec_1}

In this paper, we consider the following initial value problem for a discrete semilinear wave equation:
\begin{equation}
	\label{intro_1}
	\left\{
	\begin{array}{ll}
	\displaystyle
	u_{n+1}^{(i)}-2u_n^{(i)}+u_{n-1}^{(i)}
	=
	\delta^2\Delta_hu_n^{(i)}
	+
	\delta^2(1+n\delta)^{-q}|u_n^{(i)}|^p,
	& n\geq1,\quad i\in\Z^d,\\[1mm]
	\displaystyle
	u_0^{(i)}=\eps f_i,\qquad
	u_1^{(i)}=\eps(f_i+\delta g_i),
	& i\in\Z^d,
	\end{array}
	\right.
\end{equation}
where
\[
  d\geq1,\qquad
  p>1,\qquad
  q\geq0,\qquad
  \delta,h>0,
\]
and
\[
  \N_0:=\{0,1,2,\ldots\}.
\]
Here the discrete Laplacian is given by
\[
	\Delta_hu_n^{(i)}
	=
	\frac{1}{h^2}
	\sum_{j=1}^{d}
	\left(
	u_n^{(i-e_j)}
	-2u_n^{(i)}
	+u_n^{(i+e_j)}
	\right),
\]
where \(e_j\) denotes the \(j\)-th standard basis vector of \(\Z^d\).
We denote the solution of \eqref{intro_1} by
\[
  u^\eps
  =
  \left(
    u_n^{\eps,(i)}
  \right)_{n\in\N_0,\;i\in\Z^d}.
\]

We fix the time-step size \(\delta\) and the spatial mesh size \(h\) throughout this paper.
We consider \eqref{intro_1} as a difference equation, and do not study the limit as the mesh sizes tend to zero.
The stability and convergence of the scheme are therefore not considered, and no CFL condition is assumed.

We note that finite-step blow-up in the usual sense does not occur for \eqref{intro_1}.
Indeed, \(|u_{n}^{(i)}|^p\) is finite whenever \(u_{n}^{(i)}\) is finite.
If all the components of \(u_{n-1}\) and \(u_n\) are finite, then those of \(u_{n+1}\) are also finite.
Repeating this argument, we obtain the solution at every finite time step.
But this does not imply that its amplitude is bounded as \(n\to\infty\).

To describe the growth of the amplitude, we put
\[
  \Theta_n(u)
  :=
  \sup_{i\in\Z^d}
  |u_n^{(i)}|,
\]
and define, for \(\Lambda>0\),
\[
  N_\Lambda(u)
  :=
  \inf
  \left\{
    n\in\N_0:
    \Theta_n(u)\geq\Lambda
  \right\},
  \qquad
  \inf\emptyset:=\infty.
\]
We call \(N_\Lambda(u)\) the threshold-crossing time of \(u\).
We should note that, unlike the lifespan of a continuous blow-up solution, \(N_\Lambda(u)\) is not the maximal existence time.

In this paper, we are interested in upper bounds of \(N_\Lambda(u^\eps)\) and their dependence on the size \(\eps\) of the initial data and the threshold \(\Lambda\).
We shall also prove that every finite amplitude threshold is reached in finitely many time steps, under suitable assumptions on the initial data.

Before stating our results, we recall the continuous problem corresponding to \eqref{intro_1}:
\begin{equation}
	\label{intro_2}
	\left\{
	\begin{array}{ll}
	\displaystyle
	u_{tt}-\Delta u
	=
	(1+t)^{-q}|u|^p,
	& (t,x)\in[0,\infty)\times\R^d,\\[1mm]
	\displaystyle
	u(0,x)=\eps f(x),\qquad
	u_t(0,x)=\eps g(x),
	& x\in\R^d.
	\end{array}
	\right.
\end{equation}
If \(q=0\), then this is the usual undamped semilinear wave equation.
If \(q=p-1\), then \eqref{intro_2} is obtained from the scale-invariantly damped equation
\[
  v_{tt}-\Delta v+\frac{2}{1+t}v_t=|v|^p
\]
by the Liouville transformation
\[
  u=(1+t)v.
\]

Under the usual compact-support and positivity assumptions on the initial data, the functional
\[
  F(t)
  :=
  \int_{\R^d}u(t,x)\,dx
\]
satisfies estimates of the form
\[
  F(t)\geq C\eps t,
  \qquad
  F''(t)
  \geq
  C(1+t)^{-Q}F(t)^p,
  \qquad
  Q:=d(p-1)+q.
\]
Putting
\[
  \kappa
  :=
  p+1-Q
  =
  2-(d-1)(p-1)-q,
\]
we have the following small-data scale by the Kato-type argument:
\[
  \eps^{-(p-1)/\kappa}
\]
for \(\kappa>0\). See, for example, \cite{Kato1980,Takamura2015}.
In the case \(\kappa=0\), a logarithmic iteration gives the exponential scale
\[
  \exp\!\left(
    C\eps^{-(p-1)}
  \right).
\]
In one space dimension, these orders of the lifespan are sharp for \(0\leq q\leq2\) under an appropriate nonzero-total condition, as proved by Kitamura, Takamura and Wakasa~\cite{KitamuraTakamuraWakasa2023}.
These continuous results are mentioned only for comparison, and are not used in our proofs.

Let us recall some previous results on finite-difference approximations and discrete nonlinear wave equations.
Cho~\cite{Cho2010} studied a finite-difference scheme for nonlinear wave equations with finite-time blow-up solutions.
Saito and Sasaki~\cite{SaitoSasaki2016} studied such finite-difference solutions in more detail.
They considered convergence before the blow-up time, as well as convergence of the numerical blow-up time under suitable assumptions.

As for finite-difference schemes preserving properties of the original equation, Furihata~\cite{Furihata2001} gave two general schemes for nonlinear wave equations in one space dimension which preserve energy conservation.
Umeda, Wakasugi and Yoshikawa~\cite{UmedaWakasugiYoshikawa2022} studied energy-conserving finite-difference schemes for nonlinear wave equations with dynamic boundary conditions.
These results concern qualitative properties of the discretized equations, but not the threshold-crossing problem studied in this paper.

The results of Matsuya~\cite{Matsuya2013} and Tsubota et al.~\cite{Tsubota2025Lifespan,Tsubota2026Lifespan} are more closely related to our problem.
Matsuya considered a discrete semilinear wave equation as a difference equation and modified the nonlinear term so that it has a singular level depending on the mesh sizes.
This allows finite-step blow-up to be defined. For this equation, he proved a discrete analogue of Kato's blow-up theorem.
Later, Tsubota et al. obtained upper bounds of the lifespan in one and higher space dimensions.
Their estimates have the same order in the size of the initial data as the corresponding continuous Kato-type lifespan bounds.

There are two differences between these results and ours.
First, finite-step blow-up in \cite{Matsuya2013,Tsubota2025Lifespan,Tsubota2026Lifespan} is related to the singular level of the nonlinear term, which depends on the mesh sizes.
But \eqref{intro_1} is the standard centered discretization with a power-type nonlinear term and has no such singular level.
Thus its solution is defined at every finite time step.

Second, we consider the first time when the amplitude reaches an arbitrarily fixed finite threshold, instead of a blow-up time defined by a singularity of the recurrence.
In this way, \(N_\Lambda(u^\eps)\) is defined for the standard recurrence itself, without introducing a singularity into the nonlinear term.

We also note that the same argument applies to the nonlinear terms
\[
  (1+n\delta)^{-q}|u_n|^p
\]
with all the time weights considered here.
In particular, we can treat the undamped case \(q=0\) and the Liouville-transformed scale-invariantly damped case \(q=p-1\) by the same argument.
The growth of the support and the decay of the time weight appear together in the exponent
\[
  Q=d(p-1)+q.
\]

We are now in a position to describe our results.
Assume that the initial data have finite support and satisfy
\[
  \sum_{i\in\Z^d}f_i\geq0,
  \qquad
  \sum_{i\in\Z^d}g_i>0.
\]
We put
\[
  Q:=d(p-1)+q,
  \qquad
  \kappa
  :=
  p+1-Q
  =
  2-(d-1)(p-1)-q.
\]
Here and below, we use the notation
\[
  \log_4p
  :=
  \frac{\log p}{\log4}.
\]
When \(\kappa>0\), we obtain
\begin{equation}
	\label{intro_3}
	N_\Lambda(u^\eps)
	\leq
	C
	\left\{
		\eps^{-(p-1)/\kappa}
		+
		1+\log\log(e+\Lambda)
	\right\}.
\end{equation}
When \(\kappa=0\), we obtain
\begin{equation}
	\label{intro_4}
  N_\Lambda(u^\eps)
  \leq
  C_1
  \exp\!\left(
    C_2\eps^{-(p-1)}
  \right)
  +
  C_3
  \left\{
    1+\log(e+\Lambda)
  \right\}^{2/\log_4p}.
\end{equation}
where the constants are independent of \(\eps\in(0,1]\) and \(\Lambda>0\).

We note that, for a fixed threshold, these estimates have the same small-data order as the continuous Kato-type estimates.
Namely, for any fixed \(\Lambda>0\), there exists \(\eps_0=\eps_0(\Lambda)\in(0,1]\) such that
\begin{equation}
	\label{intro_5}
  N_\Lambda(u^\eps)
  \leq
  \begin{cases}
    \displaystyle
    C\eps^{-(p-1)/\kappa},
    & \kappa>0,\\[2mm]
    \displaystyle
    \exp\!\left(
      C\eps^{-(p-1)}
    \right),
    & \kappa=0,
  \end{cases}
\end{equation}
for \(0<\eps\leq\eps_0(\Lambda)\).
The constant \(C\) is independent of \(\eps\) and \(\Lambda\), although \(\eps_0(\Lambda)\) may depend on \(\Lambda\).

Moreover, when \(\kappa\geq0\), every finite amplitude level is reached in finitely many time steps. Hence we have
\[
  \sup_{n\in\N_0}
  \sup_{i\in\Z^d}
  |u_n^{(i)}|
  =
  \infty.
\]
Thus the solution is unbounded in time under the above assumptions, even though it exists at every finite time step.

For a fixed \(\Lambda\), the order in \(\eps\) of \eqref{intro_5} is the same as that of the Kato-type estimate for the corresponding continuous problem.
But \(N_\Lambda(u^\eps)\) is a threshold-crossing time, not the maximal existence time.
We do not claim that these estimates are sharp in general.

We briefly explain the proof. We put
\[
  U_n
  :=
  \sum_{i\in\Z^d}
  u_n^{\eps,(i)}.
\]
This sum is finite because of the finite propagation property.
Taking the sum of \eqref{intro_1} over \(i\in\Z^d\), we see that the term involving the discrete Laplacian vanishes.
Then, making use of H\"older's inequality and the finite propagation property, we obtain
\begin{equation}
\label{intro_6}
  U_n\geq C\eps n,
  \qquad
  U_{n+1}-2U_n+U_{n-1}
  \geq
  Cn^{-Q}U_n^p.
\end{equation}
On the other hand, before the amplitude reaches \(\Lambda\), we have
\begin{equation}
\label{intro_7}
  0\leq U_n\leq C\Lambda n^d.
\end{equation}
The estimates \eqref{intro_3} and \eqref{intro_4} are proved by comparing \eqref{intro_7} with lower bounds derived from \eqref{intro_6}.
For \(\kappa>0\), we use a discrete Kato-type iteration to obtain the lower bounds, while for \(\kappa=0\), we use a logarithmic iteration.

This paper is organized as follows.
In Section~\ref{sec_2}, we introduce the time-weighted discrete model and define the threshold-crossing time.
In Section~\ref{sec_3}, we state our main results.
In Section~\ref{sec_4}, we show the finite propagation estimate and derive the scalar difference inequalities.
The discrete iteration lemmas are proved in Section~\ref{sec_5}.
Section~\ref{sec_6} is devoted to the proof of the main results.
Finally, we discuss some further problems in Section~\ref{sec_7}.

In what follows, \(C\) denotes a positive constant which may change from line to line.
Unless otherwise stated, it is independent of \(\eps\) and \(\Lambda\).

\section{Time-weighted semilinear discretization}
\label{sec_2}

In this section, we introduce the discrete equation and the threshold-crossing time.
We take the coefficient of the nonlinear term to be
\[
	\rho(t)=(1+t)^{-q},
	\qquad q\geq 0,
\]
so that \(q=0\) and \(q=p-1\) correspond to the undamped equation and the transformed scale-invariantly damped equation, respectively.
The finite propagation property and the scalar difference estimates are proved in the same way for both equations, with only the value of \(q\) changed.

\subsection{The continuous problem and the space--time grid}

Let \(d\geq 1\), \(p>1\), and \(q\geq 0\).
We consider the following initial value problem for a time-weighted semilinear wave equation:
\begin{equation}
	\begin{cases}
	u_{tt}-\Delta u
	=
	\rho(t)|u|^p,
	& (t,x)\in[0,\infty)\times\mathbb{R}^d,\\[1mm]
	u(0,x)=f(x),\qquad
	u_t(0,x)=g(x),
	& x\in\mathbb{R}^d,
	\end{cases}
	\label{eq_3}
\end{equation}
where
\[
	\rho(t):=(1+t)^{-q},
	\qquad t\geq 0.
\]
Here \(q\) gives the decay rate of the nonlinear coefficient.
We are particularly interested in the following two cases.
If \(q=0\), then \(\rho(t)\equiv 1\), and \eqref{eq_3} becomes the undamped semilinear wave equation
\[
	u_{tt}-\Delta u=|u|^p.
\]
If \(q=p-1\), then $\rho(t)=(1+t)^{-(p-1)}$, and \eqref{eq_3} becomes
\[
	u_{tt}-\Delta u = \frac{|u|^p}{(1+t)^{p-1}}.
\]
The latter equation is obtained from the scale-invariantly damped semilinear wave equation with damping parameter \(\mu=2\) by the Liouville transformation.

Let \(\delta>0\) and \(h>0\) denote the time-step size and the spatial mesh size, respectively. We put
\[
	t_n:=n\delta, \qquad x_i:=hi, \qquad
	n\in\mathbb{N}_0,\quad i\in\mathbb{Z}^d,
\]
where $\mathbb{N}_0:=\{0,1,2,\ldots\}$.
Here \(e_j\) is the \(j\)-th standard basis vector of \(\mathbb{Z}^d\),
and \(u_n^{(i)}\) denotes an approximation to \(u(t_n,x_i)\).

For a lattice function
\[
	w=(w^{(i)})_{i\in\mathbb{Z}^d},
\]
we define the discrete Laplacian by
\[
	\Delta_h w^{(i)}
	:=
	\sum_{j=1}^{d}
	\frac{
		w^{(i-e_j)}-2w^{(i)}+w^{(i+e_j)}
	}{h^2},
	\qquad
	i\in\mathbb{Z}^d.
\]

\subsection{The time-weighted discrete model}

We put
\[
	\rho_n
	:= \rho(t_n)
	= (1+n\delta)^{-q},
	\qquad
	n\in\mathbb{N}_0.
\]
Applying the standard centered finite-difference scheme to \eqref{eq_3}, we obtain
\[
	\frac{
		u_{n+1}^{(i)}-2u_n^{(i)}+u_{n-1}^{(i)}
	}{\delta^2}
	=
	\Delta_hu_n^{(i)}
	+
	\rho_n|u_n^{(i)}|^p,
	\qquad
	n\geq1,\quad i\in\mathbb{Z}^d.
\]
This is equivalent to
\begin{equation}
	u_{n+1}^{(i)}-2u_n^{(i)}+u_{n-1}^{(i)}
	=
	\delta^2\Delta_hu_n^{(i)}
	+
	\delta^2\rho_n|u_n^{(i)}|^p,
	\qquad
	n\geq1,\quad i\in\mathbb{Z}^d.
	\label{eq_4}
\end{equation}
Solving \eqref{eq_4} for the value at the next time step, we can write it in the form
\begin{equation}
	u_{n+1}^{(i)}
	= \frac{\delta^2}{h^2} \sum_{j=1}^{d} \left(u_n^{(i-e_j)}+u_n^{(i+e_j)}\right)
	+ \left(2-\frac{2d\delta^2}{h^2}\right)u_n^{(i)}
	- u_{n-1}^{(i)}
	+ \delta^2(1+n\delta)^{-q}|u_n^{(i)}|^p,
	\label{eq_5}
\end{equation}
for \(n\geq1\) and \(i\in\mathbb{Z}^d\).

We take the initial values as
\begin{equation}
	u_0^{(i)}=f(x_i),
	\qquad
	u_1^{(i)}=f(x_i)+\delta g(x_i),
	\qquad
	i\in\mathbb{Z}^d.
	\label{eq_6}
\end{equation}
Then we have
\[
	\frac{u_1^{(i)}-u_0^{(i)}}{\delta} = g(x_i),
\]
so that the second condition in \eqref{eq_6} gives the usual forward-difference approximation of the initial velocity.

When the initial data are multiplied by a small parameter \(\eps>0\), we take
\[
	u_0^{\eps,(i)} = \eps f(x_i),
	\qquad
	u_1^{\eps,(i)} = \eps\bigl(f(x_i)+\delta g(x_i)\bigr),
	\qquad
	i\in\mathbb{Z}^d,
\]
and denote the corresponding solution of \eqref{eq_5} by
\[
	u^\eps = \left(u_n^{\eps,(i)}\right)_{n\in\mathbb{N}_0,\;i\in\mathbb{Z}^d}.
\]

In the scalar difference inequality derived below, the decay of the time weight and the growth of the support give the exponent
\begin{equation}
	Q:=d(p-1)+q.
	\label{eq_7}
\end{equation}
In particular, we have
\begin{equation*}
	Q=
	\begin{cases}
	d(p-1),
	& q=0,\\[1mm]
	(d+1)(p-1),
	& q=p-1.
	\end{cases}
\end{equation*}
Thus we distinguish the two equations by the value of \(q\).

We note again that, in contrast to a differential equation, \eqref{eq_5} has no pointwise divergence at a finite time step.
Indeed, if \(u_{n-1}\) and \(u_n\) have finite components, then \(u_{n+1}\) given by \eqref{eq_5} also has finite components.
By induction, we can continue the solution to every finite time step, although its amplitude may become arbitrarily large as \(n\to\infty\).

Hence, instead of the lifespan defined by continuation of the solution, we consider the first time at which the amplitude reaches a given level.
We introduce this time as follows.

\subsection{Threshold-crossing time}

Let
\[
	u = \left(u_n^{(i)}\right)_{n\in\mathbb{N}_0,\;i\in\mathbb{Z}^d}
\]
be a solution of \eqref{eq_5}.
We define the maximum amplitude at the \(n\)-th time step by
\[
	\Theta_n(u)
	:=
	\sup_{i\in\mathbb{Z}^d}
	|u_n^{(i)}|,
	\qquad
	n\in\mathbb{N}_0.
\]
We fix an amplitude threshold \(\Lambda>0\).

\begin{definition}[Threshold-crossing time]
	\label{def_2_1}
Let \(u\) be a solution of \eqref{eq_5}, and let \(\Lambda>0\).
We define the discrete threshold-crossing time of \(u\) at the level \(\Lambda\) by
\[
	N_{\Lambda}(u)
	:=
	\inf
	\left\{
		n\in\mathbb{N}_0:
		\Theta_n(u)\geq\Lambda
	\right\},
	\qquad
	\inf\varnothing:=\infty.
\]
If \(N_{\Lambda}(u)<\infty\), we say that the solution crosses the amplitude threshold \(\Lambda\) in finitely many time steps.
\end{definition}

It follows from the definition that, for \(n<N_\Lambda(u)\),
\[
	|u_n^{(i)}|<\Lambda
	\qquad
	\text{for every }i\in\mathbb{Z}^d.
\]
Making use of this estimate and the finite propagation property, we obtain polynomial upper bounds of suitable spatial sums as long as the threshold is not reached.
We shall compare these upper bounds with the lower bounds obtained by iteration in the following sections.

\begin{remark}
We note that \(N_{\Lambda}(u)\) is not the maximal existence time, since \eqref{eq_5} is well defined at every finite time step.
The condition \(N_{\Lambda}(u)<\infty\) means that the amplitude reaches \(\Lambda\), and does not mean that the recurrence becomes singular or that the solution ceases to exist there.
\end{remark}

\begin{remark}
If \(0<\Lambda_1\leq\Lambda_2\), then
\[
	N_{\Lambda_1}(u) \leq N_{\Lambda_2}(u).
\]
Namely, the threshold-crossing time is nondecreasing with respect to the amplitude threshold.
\end{remark}

\section{Main results}
\label{sec_3}

Let \(d\geq1\), \(p>1\), and \(q\geq0\), and let
\[
	f=(f_i)_{i\in\Z^d},
	\qquad
	g=(g_i)_{i\in\Z^d}
\]
be real-valued sequences which satisfy
\begin{equation}
	\begin{gathered}
		\supp f\cup\supp g
		\subset
		\left\{
			i\in\Z^d:\|i\|_1\leq R
		\right\},
		\\[1mm]
		F_0:=
		\sum_{i\in\Z^d}f_i
		\geq0,
		\qquad
		G_0:=
		\sum_{i\in\Z^d}g_i
		>0,
	\end{gathered}
	\label{eq_8}
\end{equation}
for some \(R\geq0\), where
\[
	\|i\|_1:=|i_1|+\cdots+|i_d|.
\]
Since \(f\) and \(g\) have finite support, the sums in \eqref{eq_8} are finite.

For \(0<\eps\leq1\), we consider the initial data
\begin{equation}
	u_0^{\eps,(i)}
	=
	\eps f_i,
	\qquad
	u_1^{\eps,(i)}
	=
	\eps\bigl(f_i+\delta g_i\bigr),
	\qquad
	i\in\Z^d,
	\label{eq_9}
\end{equation}
and denote by
\[
	u^\eps = \left(u_n^{\eps,(i)}\right)_{n\in\N_0,\;i\in\Z^d}
\]
the corresponding solution of \eqref{eq_5}. We note that
\[
	\frac{u_1^{\eps,(i)}-u_0^{\eps,(i)}}{\delta} = \eps g_i,
\]
so that \(g\) gives the discretized initial velocity.

For each \(\Lambda>0\), let \(N_\Lambda(u^\eps)\) be as in Definition~\ref{def_2_1}.
If either of the initial states already reaches the threshold, then
\[
	N_\Lambda(u^\eps)\in\{0,1\}.
\]
In this case, we obtain the estimates below by taking the constants sufficiently large.
In the proof of the main theorem, we may therefore assume that
\[
	\sup_{i\in\Z^d}|u_0^{\eps,(i)}| < \Lambda,
	\qquad
	\sup_{i\in\Z^d}|u_1^{\eps,(i)}| < \Lambda.
\]

We recall the definition \eqref{eq_7}:
\[
	Q = d(p-1) + q.
\]
We also put
\[
	\kappa
	:= p+1-Q
	= 2-(d-1)(p-1)-q.
\]
The subcritical and critical cases are distinguished by the sign of \(\kappa\) in the scalar difference inequality, which will be derived in Section~\ref{sec_4}.

\begin{thm}[Threshold-crossing estimate for the time-weighted model]
	\label{thm_3_1}
Let \(d\geq1\), \(p>1\), and \(q\geq0\), and fix \(\delta>0\) and \(h>0\).
Assume that the initial sequences \(f\) and \(g\) satisfy \eqref{eq_8}.
Then we have the following estimates.

\medskip

\noindent
\textnormal{(i) Subcritical case.}
Assume that
\[
	\kappa = p+1-d(p-1)-q > 0.
\]
Then there exists a positive constant \(C>0\), independent of
\(\eps\in(0,1]\) and \(\Lambda>0\), such that
\begin{equation}
	N_\Lambda(u^\eps)
	\leq
	C
	\left\{
		\eps^{-\frac{p-1}{\kappa}}
		+
		1+\log\log(e+\Lambda)
	\right\}.
	\label{eq_10}
\end{equation}

\medskip

\noindent
\textnormal{(ii) Critical case.}
Assume that
\[
	\kappa=0,
	\qquad\text{equivalently}\qquad
	d(p-1)+q=p+1.
\]
Then there exist positive constants \(C_1,C_2,C_3>0\), independent of
\(\eps\in(0,1]\) and \(\Lambda>0\), such that
\begin{equation}
	N_\Lambda(u^\eps)
	\leq
	C_1
	\exp\left(
		C_2\eps^{-(p-1)}
	\right)
	+
	C_3
	\left\{
		1+\log(e+\Lambda)
	\right\}^{\frac{2}{\log_4p}}.
	\label{eq_11}
\end{equation}
Taking the constants larger if necessary, we also have a positive constant \(C>0\), independent of \(\eps\) and \(\Lambda\), such that
\begin{equation}
	N_\Lambda(u^\eps)
	\leq
	C
	\exp\left(
		C\eps^{-(p-1)}
	\right)
	\left\{
		1+\log(e+\Lambda)
	\right\}^{\frac{2}{\log_4p}}.
	\label{eq_12}
\end{equation}
\end{thm}

\begin{remark}
The constants in Theorem~\ref{thm_3_1} may depend on
\[
	d,\quad p,\quad
	q,\quad \delta,\quad
	h,\quad R,\quad
	f,\quad \text{and}\quad g,
\]
which are fixed throughout the proof.
They are independent of \(\eps\in(0,1]\) and \(\Lambda>0\), and the dependence of the upper bounds on \(\eps\) and \(\Lambda\) is given explicitly.
\end{remark}

\begin{remark}
The two cases in the iteration are determined by
\[
	\kappa = p+1-\bigl(d(p-1)+q\bigr).
\]
Here the term \(d(p-1)\) comes from the growth of the support, while \(q\) comes from the decay of \(\rho_n\).
If we increase \(q\), then the nonlinear growth in the scalar inequality becomes weaker, and the range of parameters satisfying \(\kappa>0\) becomes smaller.
\end{remark}

\begin{remark}
The condition \(\kappa=0\) is critical for the discrete iteration used in this paper.
We do not assert that it is the sharp critical condition which separates bounded and unbounded solutions of the full discrete system over all time steps.
\end{remark}

\begin{remark}
In Theorem~\ref{thm_3_1}, we estimate the number of time steps needed for the amplitude to reach \(\Lambda\).
If we use physical time, the corresponding threshold-crossing time is given by
\[
	T_\Lambda(u^\eps)
	:=
	\delta N_\Lambda(u^\eps).
\]
Since \(\delta>0\) is fixed, the upper bounds of \(T_\Lambda(u^\eps)\) are given by the right hand sides of \eqref{eq_10}--\eqref{eq_12} multiplied by \(\delta\).
We do not assert that the estimates are uniform with respect to \(\delta\) and \(h\).
\end{remark}

\begin{remark}
The double logarithm in the subcritical estimate is obtained as follows.
As long as the amplitude is less than \(\Lambda\), the threshold condition and the finite propagation property give $F_n\leq K\Lambda n^s$ for a suitable nonnegative sequence \(F_n\).
On the other hand, the subcritical iteration gives a lower bound whose main factor grows as
\[
	2^{p^{cn}}
\]
for some \(c>0\).
Comparing these upper and lower bounds, we obtain $1+\log\log(e+\Lambda)$ in \eqref{eq_10}.
\end{remark}

\begin{remark}
In the critical case, the logarithmic iteration gives a lower bound with the main factor
\[
	2^{cn^\gamma},
	\qquad
	\gamma=\frac{1}{2}\log_4p.
\]
Comparing this factor with the upper bound $F_n\leq K\Lambda n^s$, we get
\[
	\left\{1+\log(e+\Lambda)\right\}^{1/\gamma}
	=
	\left\{1+\log(e+\Lambda)\right\}^{\frac{2}{\log_4p}}.
\]
In \eqref{eq_11}, the two terms depending on the size of the initial data and on the threshold are added, rather than multiplied as in \eqref{eq_12}.
The former estimate is therefore stronger than the latter.
We do not claim that either the order in \(\eps\) or that in \(\Lambda\) is optimal in the critical case.
\end{remark}

\begin{cor}[Fixed-threshold small-data estimates]
	\label{cor_fixed_threshold}
Let the assumptions of Theorem~\ref{thm_3_1} hold.

\medskip

\noindent
\textnormal{(i)}
Assume that
\[
	\kappa=p+1-d(p-1)-q>0.
\]
Then there exists a positive constant \(C>0\), independent of \(\eps\) and \(\Lambda\), with the following property.
For any fixed \(\Lambda>0\), there exists \(\eps_0(\Lambda)\in(0,1]\) such that
\[
	N_\Lambda(u^\eps)
	\le
	C
	\eps^{-\frac{p-1}{\kappa}}
\]
for \(0<\eps\le\eps_0(\Lambda)\).

\medskip

\noindent
\textnormal{(ii)}
Assume that
\[
	\kappa=0.
\]
Then there exists a positive constant \(C>0\), independent of \(\eps\) and \(\Lambda\), with the following property.
For any fixed \(\Lambda>0\), there exists \(\eps_0(\Lambda)\in(0,1]\) such that
\[
	N_\Lambda(u^\eps)
	\le
	\exp\!\left(
		C\eps^{-(p-1)}
	\right)
\]
for \(0<\eps\le\eps_0(\Lambda)\).
\end{cor}

\begin{proof}
(i) We put
\[
	\ell_\Lambda
	:=
	1+\log\log(e+\Lambda)
	\qquad\text{and}\qquad
	\eps_0(\Lambda)
	:=
	\min
	\left\{
		1,\,
		\ell_\Lambda^{-\frac{\kappa}{p-1}}
	\right\}.
\]
Then, for \(0<\eps\leq\eps_0(\Lambda)\), we obtain
\[
	\ell_\Lambda
	\leq
	\eps^{-\frac{p-1}{\kappa}}.
\]
Making use of this inequality and the subcritical estimate in Theorem~\ref{thm_3_1}, we obtain
\[
	N_\Lambda(u^\eps)
	\leq
	2C
	\eps^{-\frac{p-1}{\kappa}}.
\]
Changing the constant, we obtain (i).

\medskip

(ii) We put
\[
	m_\Lambda := 1+\log(e+\Lambda)
	\qquad\text{and}\qquad
	r := \frac{2}{\log_4p}.
\]
Let \(C_1,C_2,C_3>0\) be the constants in the additive estimate of Theorem~\ref{thm_3_1}.
We take \(\eps_0(\Lambda)\in(0,1]\) sufficiently small so that, for \(0<\eps\le\eps_0(\Lambda)\), we have
\[
	m_\Lambda^r
	\leq
	\exp\!\left(
		C_2\eps^{-(p-1)}
	\right).
\]
Then we obtain
\[
	N_\Lambda(u^\eps)
	\leq (C_1+C_3) \exp\!\left(C_2\eps^{-(p-1)}\right).
\]

Since \(0<\eps\le1\), we have \(\eps^{-(p-1)}\ge1\), which gives us
\[
	(C_1+C_3) \leq \exp\!\left((\log(C_1+C_3))_+ \eps^{-(p-1)}\right),
\]
where
\[
	(\log(C_1+C_3))_+
	:= \max\{\log(C_1+C_3), 0\}.
\]
Making use of the above inequalities, we obtain
\[
	N_\Lambda(u^\eps)
	\leq \exp\!\left(C \eps^{-(p-1)}\right),
\]
where
\[
	C := C_2 + (\log(C_1+C_3))_+.
\]
This completes the proof.
\end{proof}

\subsection{Consequences for the two principal models}

Applying Theorem~\ref{thm_3_1} to the undamped equation and the transformed scale-invariantly damped equation, we obtain the following corollaries.

\begin{cor}[Undamped discrete semilinear wave equation]
	\label{cor_undamped}
Let \(q=0\). Then we have
\[
	Q = d(p-1),
	\qquad
	\kappa = 2-(d-1)(p-1).
\]

\medskip

\noindent
\textnormal{(i)}
Let \(d=1\). In this case, \(\kappa=2>0\) holds for any \(p>1\).
Then there exists a positive constant \(C>0\),
independent of \(\eps\in(0,1]\) and \(\Lambda>0\), such that
\[
	N_\Lambda(u^\eps)
	\leq
	C
	\left\{
		\eps^{-\frac{p-1}{2}} + 1+\log\log(e+\Lambda)
	\right\}.
\]

\medskip

\noindent
\textnormal{(ii)}
Let \(d\geq2\), and assume that $1 < p < (d+1)/(d-1)$.
Then there exists a positive constant \(C>0\),
independent of \(\eps\in(0,1]\) and \(\Lambda>0\), such that
\[
	N_\Lambda(u^\eps)
	\leq
	C
	\left\{
		\eps^{-\frac{p-1}{2-(d-1)(p-1)}} + 1+\log\log(e+\Lambda)
	\right\}.
\]

\medskip

\noindent
\textnormal{(iii)}
Let \(d\geq2\), and assume that $p=(d+1)/(d-1)$.
Then there exist positive constants \(C_1,C_2,C_3>0\),
independent of \(\eps\in(0,1]\) and \(\Lambda>0\), such that
\[
	N_\Lambda(u^\eps)
	\leq C_1 \exp\left(C_2\eps^{-(p-1)}\right)
	+ C_3\left\{1+\log(e+\Lambda)\right\}^{\frac{2}{\log_4p}}.
\]
\end{cor}

\begin{remark}
For \(d\geq2\), the value
\[
	p_{\mathrm K}(d)
	:= \frac{d+1}{d-1}
	= 1+\frac{2}{d-1}
\]
is critical for the iteration based on the spatial sum.
But the sharp critical exponent of the undamped discrete semilinear wave equation is not determined by Corollary~\ref{cor_undamped}.
\end{remark}

\begin{cor}[Transformed scale-invariantly damped equation]
	\label{cor_scale_invariant}
Let \(q=p-1\). Then we have
\[
	Q = (d+1)(p-1),
	\qquad
	\kappa = 2-d(p-1).
\]

\medskip

\noindent
\textnormal{(i)}
Assume that $1 < p < 1+2/d$.
Then there exists a positive constant \(C>0\),
independent of \(\eps\in(0,1]\) and \(\Lambda>0\), such that
\[
	N_\Lambda(u^\eps)
	\leq
	C
	\left\{
		\eps^{-\frac{p-1}{2-d(p-1)}} + 1+\log\log(e+\Lambda)
	\right\}.
\]

\medskip

\noindent
\textnormal{(ii)}
Assume that $p = 1 + 2/d$.
Then there exist positive constants \(C_1,C_2,C_3>0\),
independent of \(\eps\in(0,1]\) and \(\Lambda>0\), such that
\[
	N_\Lambda(u^\eps)
	\leq C_1 \exp\left(C_2\eps^{-(p-1)}\right)
	+ C_3 \left\{1+\log(e+\Lambda)\right\}^{\frac{2}{\log_4p}}.
\]
\end{cor}

\subsection{Finite crossing of every amplitude threshold}

\begin{cor}[Finite crossing of every amplitude threshold]
	\label{cor_3_3}
Let \(d\geq1\), \(p>1\), and \(q\geq0\), and fix \(\delta>0\) and \(h>0\).
Assume that
\[
	\kappa = p+1-d(p-1)-q
	\geq 0.
\]
Let
\[
	u = \left(u_n^{(i)}\right)_{n\in\N_0,\;i\in\Z^d}
\]
be a solution of \eqref{eq_5}.
Suppose that, for some \(R\geq0\), we have
\begin{equation}
	\supp u_0\cup\supp u_1
	\subset \left\{i\in\Z^d:\|i\|_1\leq R\right\},
	\label{eq_13}
\end{equation}
and that
\begin{equation}
	\sum_{i\in\Z^d}u_0^{(i)}
	\geq0,
	\qquad
	\sum_{i\in\Z^d}
	\left(
		u_1^{(i)}-u_0^{(i)}
	\right)
	>0.
	\label{eq_14}
\end{equation}
Then, for any \(\Lambda>0\), we have
\[
	N_\Lambda(u)<\infty.
\]
\end{cor}

\begin{remark}
We can apply Theorem~\ref{thm_3_1} in Corollary~\ref{cor_3_3} by putting
\[
	f_i:=u_0^{(i)},
	\qquad
	g_i:=
	\delta^{-1}
	\left(
		u_1^{(i)}-u_0^{(i)}
	\right),
	\qquad
	\eps:=1.
\]
It follows from \eqref{eq_13} that \(f\) and \(g\) have finite support in the same lattice ball.
By \eqref{eq_14}, we also have
\[
	\sum_{i\in\Z^d}f_i\geq0,
	\qquad
	\sum_{i\in\Z^d}g_i>0.
\]
Hence all the assumptions of Theorem~\ref{thm_3_1} are satisfied.
\end{remark}

\begin{remark}
We note again that Corollary~\ref{cor_3_3} does not mean a finite-step singularity of the recurrence.
The equation \eqref{eq_5} is well defined at every finite time step, but its amplitude cannot be bounded by a finite constant for all \(n\in\N_0\).
\end{remark}

\section{Preliminary estimates}
\label{sec_4}

In this section, we shall prove the preliminary estimates for the main results.
We first show the finite propagation property and estimate the number of lattice points in the support.
We then consider the spatial sum \(U_n\) and derive a scalar second-order difference inequality.
In this inequality, the growth of the support gives \(n^{-d(p-1)}\), while the time-dependent coefficient gives \(n^{-q}\).
The product of these factors is expressed by the exponent
\[
	Q=d(p-1)+q
\]
defined in \eqref{eq_7}.
We shall use these estimates to apply the iteration lemmas of Section~\ref{sec_5}.

\subsection{Finite propagation and auxiliary estimates}

We first give an estimate for the number of lattice points.

\begin{lemma}
	\label{lem_4_1}
Let $d\geq1$ and $r\in\N$. Then we have
\[
	\card\{i\in\Z^d:\|i\|_1\leq r\}\leq2^{2d+1}(1+r)^d.
\]
\end{lemma}

\begin{proof}
We note that, for $m\geq1$, the number of lattice points satisfying $\|i\|_1=m$ is at most $2^{2d}(1+m)^{d-1}$.
Taking the sum for $0\leq m\leq r$, we obtain the desired estimate.
\end{proof}

\begin{lemma}[Discrete summation formula]
	\label{lem_4_2}
Let \((F_n)\) be a real-valued sequence, and let \(1\leq M<n\).
Then we have
\begin{align}
	F_n
	={}&
	F_M
	+(n-M)(F_M-F_{M-1})
	+
	\sum_{m=M}^{n-1}
	(n-m)
	\bigl(
		F_{m+1}-2F_m+F_{m-1}
	\bigr).
	\label{eq_15}
\end{align}
In particular, if $F_M\geq0$, $F_M-F_{M-1}\geq0$,
and $F_{m+1}-2F_m+F_{m-1}\geq H_m\geq0$, we obtain
\[
	F_n\geq\sum_{m=M}^{n-1}(n-m)H_m.
\]
\end{lemma}

\begin{proof}
We put $D_n=F_n-F_{n-1}$.
Noticing that $D_{m+1}-D_m=F_{m+1}-2F_m+F_{m-1}$, we obtain \eqref{eq_15} by summing first with respect to $m$ and then with respect to the time index.
\end{proof}

\subsection{Scalar difference inequalities}
\label{subsec_4_2}

For the solution \(u^\eps\), we define the following functional:
\[
	U_n := \sum_{i\in\Z^d}u_n^{\eps,(i)},
	\qquad
	n \in \N_0.
\]

\begin{prop}[Scalar difference inequalities]
	\label{prop_4_3}
Let \(d\geq1\), \(p>1\), and \(q\geq0\), and fix \(\delta>0\) and \(h>0\).
Assume that the initial sequences \(f\) and \(g\) satisfy \eqref{eq_8}.
Let \(u^\eps\) be the solution of \eqref{eq_5} with the initial data \eqref{eq_9}.
Then there exist positive constants \(A,B,K>0\),
independent of \(\eps\in(0,1]\) and \(\Lambda>0\),
such that the following estimates hold for every \(\Lambda>0\) and $1\leq n<N_\Lambda(u^\eps)$:
\begin{align}
	&U_n-U_{n-1}>0,
	\qquad
	0\leq U_n\leq K\Lambda n^d,
	\label{eq_16}
	\\
	&U_n\geq A\eps n,
	\label{eq_17}
	\\
	&U_{n+1}-2U_n+U_{n-1}
	\geq
	Bn^{-Q}U_n^p,
	\qquad
	Q=d(p-1)+q.
	\label{eq_18}
\end{align}
\end{prop}

\begin{proof}
We first show the finite propagation property.
From \eqref{eq_9} and \eqref{eq_8}, we have
\[
	\supp u_0^\eps\cup\supp u_1^\eps
	\subset
	\left\{
		i\in\Z^d:\|i\|_1\leq R
	\right\}.
\]
Since \eqref{eq_5} involves only the values at \(i\) and at its nearest neighbors \(i\pm e_j\), induction on \(n\) gives us
\[
	\supp u_n^\eps
	\subset
	\left\{
		i\in\Z^d:
		\|i\|_1\leq R+n-1
	\right\},
	\qquad
	n\geq1.
\]
Thus all the spatial sums below are finite.
By Lemma~\ref{lem_4_1}, we obtain, for \(n\geq1\),
\begin{equation}
	\card(\supp u_n^\eps)
	\leq 2^{2d+1}(R+n)^d
	\leq C_{\mathrm{supp}}n^d,
	\label{eq_19}
\end{equation}
where we may take $C_{\mathrm{supp}} := 2^{2d+1}(1+R)^d$.

Next we sum \eqref{eq_4} over \(i\in\Z^d\).
Since \(u_n^\eps\) has finite support, we have
\begin{align*}
	\sum_{i\in\Z^d}
	\Delta_hu_n^{\eps,(i)}
	&=
	\frac{1}{h^2} \sum_{j=1}^{d} \sum_{i\in\Z^d}
	\left(
		u_n^{\eps,(i-e_j)}-2u_n^{\eps,(i)}+u_n^{\eps,(i+e_j)}
	\right)
	\\
	&=
	\frac{1}{h^2} \sum_{j=1}^{d} \left(U_n-2U_n+U_n\right)
	=0.
\end{align*}
Here we have used the changes of indices $k=i-e_j$ and $k=i+e_j$ in the sums.
Therefore we get
\begin{equation}
	U_{n+1}-2U_n+U_{n-1}
	=
	\delta^2(1+n\delta)^{-q}
	\sum_{i\in\Z^d}
	|u_n^{\eps,(i)}|^p,
	\qquad
	n\geq1.
	\label{eq_20}
\end{equation}

The right hand side of \eqref{eq_20} is nonnegative. Thus \(U_n-U_{n-1}\) is nondecreasing with respect to \(n\).
From the initial data \eqref{eq_9}, we also have
\[
	U_0 = \eps F_0 \geq0
	\qquad
	\text{and}
	\qquad
	U_1-U_0 = \eps\delta G_0 >0.
\]
Hence, for \(n\geq1\), we have
\begin{equation}
	U_n-U_{n-1}
	\geq
	U_1-U_0
	=
	\eps\delta G_0
	>0.
	\label{eq_21}
\end{equation}
Summing \eqref{eq_21} from \(1\) to \(n\), we obtain
\[
	U_n
	= U_0 + \sum_{m=1}^{n} (U_m-U_{m-1})
	\geq \eps F_0 + n\eps\delta G_0
	\geq \eps\delta G_0\,n.
\]
This proves \eqref{eq_17} with $A:=\delta G_0$.

Next we shall estimate the nonlinear term.
H\"older's inequality gives us
\[
	|U_n|
	=
	\left|
		\sum_{i\in\supp u_n^\eps}
		u_n^{\eps,(i)}
	\right|
	\leq
	\left(
		\sum_{i\in\Z^d}
		|u_n^{\eps,(i)}|^p
	\right)^{1/p}
	\left(
		\card(\supp u_n^\eps)
	\right)^{(p-1)/p}.
\]
Taking the \(p\)-th power of both sides and making use of \eqref{eq_19}, we get
\begin{equation}
	\sum_{i\in\Z^d} |u_n^{\eps,(i)}|^p
	\geq C_{\mathrm{supp}}^{-(p-1)} n^{-d(p-1)} |U_n|^p.
	\label{eq_22}
\end{equation}

For \(n\geq1\), we have $1+n\delta \leq (1+\delta)n$.
Taking account of \(q\geq0\), we get
\begin{equation}
	(1+n\delta)^{-q}
	\geq
	(1+\delta)^{-q}n^{-q}.
	\label{eq_23}
\end{equation}
It follows from \eqref{eq_20}, \eqref{eq_22} and \eqref{eq_23}, together with \(U_n\geq0\), that
\begin{align*}
	U_{n+1}-2U_n+U_{n-1}
	&\geq \delta^2 (1+\delta)^{-q} C_{\mathrm{supp}}^{-(p-1)} n^{-d(p-1)-q} U_n^p\\
	&= \delta^2 (1+\delta)^{-q} C_{\mathrm{supp}}^{-(p-1)} n^{-Q} U_n^p.
\end{align*}
Hence we obtain \eqref{eq_18} by setting
\[
	Q:=d(p-1)+q
	\qquad
	\text{and}
	\qquad
	B := \delta^2 (1+\delta)^{-q} C_{\mathrm{supp}}^{-(p-1)}.
\]

Finally, let $n<N_\Lambda(u^\eps)$.
By Definition~\ref{def_2_1}, we have
\[
	|u_n^{\eps,(i)}| < \Lambda
	\qquad
	\text{for every }i\in\Z^d.
\]
Making use of \(U_n\geq0\) and \eqref{eq_19}, we get
\begin{align*}
	0 \leq U_n
	\leq \sum_{i\in\Z^d} |u_n^{\eps,(i)}|
	\leq \Lambda\, \card(\supp u_n^\eps)
	\leq C_{\mathrm{supp}} \Lambda n^d.
\end{align*}
Hence \eqref{eq_16} holds with \(K:=C_{\mathrm{supp}}\). This completes the proof.
\end{proof}

\section{Discrete iteration lemmas}
\label{sec_5}

In this section, we shall show two iteration lemmas for nonnegative sequences.
We shall apply these lemmas in Section~\ref{sec_6} to the functional \(U_n\) of Proposition~\ref{prop_4_3}.

First we consider the subcritical case
\[
	\kappa := p+1-Q > 0
\]
by a discrete Kato-type iteration. Next we consider the critical case
\[
	\kappa = 0,
	\qquad\text{equivalently}\qquad
	Q = p+1,
\]
by a logarithmic iteration.

\subsection{Subcritical discrete Kato iteration}
\label{subsec_5_1}

\begin{lemma}[Subcritical discrete Kato iteration]
	\label{lem_5_1}
Let \(p>1\) and \(Q\geq0\), and define $\kappa:=p+1-Q$.
Assume that \(\kappa>0\).
Let \(N\in\N\) and \(L\in\N\cup\{\infty\}\) satisfy \(N<L\).
Let \((F_n)\) be a nonnegative sequence defined for $N-1\leq n\leq L$ when \(L<\infty\),
and for \(n\geq N-1\) when \(L=\infty\).
Suppose that there exist constants $A,B,K>0$, $s>0$, $0<\eps\leq1$, and $\Lambda>0$ such that
\begin{align}
	&F_n-F_{n-1}\geq0,
	\qquad
	F_n\leq K\Lambda n^s,
	\label{eq_24}\\
	&F_n\geq A\eps(n-N),
	\label{eq_25}\\
	&F_{n+1}-2F_n+F_{n-1} \geq Bn^{-Q}F_n^p
	\label{eq_26}
\end{align}
for \(N\leq n<L\).
Then there exists a positive constant \(C>0\),
independent of \(\eps\in(0,1]\) and \(\Lambda>0\), such that
\begin{equation}
	L \leq C \left\{\eps^{-\frac{p-1}{\kappa}} + 1+\log\log(e+\Lambda)\right\}.
	\label{eq_27}
\end{equation}
\end{lemma}

\begin{proof}
We first derive a lower bound for the sequence. We put
\[
	\alpha_0 := 0,
	\qquad
	\beta_0 := 1,
	\qquad
	A_0 := A,
\]
and define the sequences by
\begin{align}
	\alpha_{j+1}
	&:= p\alpha_j+Q,
	\label{eq_28}
	\\
	\beta_{j+1}
	&:= p\beta_j+2,
	\label{eq_29}
	\\
	A_{j+1}
	&:= \frac{BA_j^p} {(p\beta_j+1)(p\beta_j+2)}.
	\label{eq_30}
\end{align}
We are going to show that
\begin{equation}
	F_n
	\geq
	A_j\eps^{p^j}
	n^{-\alpha_j}
	(n-N-j)^{\beta_j},
	\qquad
	N+j\leq n<L,
	\label{eq_31}
\end{equation}
for all \(j\geq0\).

For \(j=0\), \eqref{eq_31} is nothing but \eqref{eq_25}.
Suppose that \eqref{eq_31} is true for some \(j\geq0\).
Then, for \(m\geq N+j\), we get from the induction hypothesis that
\[
	F_m^p \geq A_j^p\eps^{p^{j+1}} m^{-p\alpha_j} (m-N-j)^{p\beta_j}.
\]
Putting this estimate into \eqref{eq_26}, we obtain
\[
	F_{m+1}-2F_m+F_{m-1}
	\geq
	BA_j^p\eps^{p^{j+1}}
	m^{-Q-p\alpha_j}
	(m-N-j)^{p\beta_j}.
\]

Since \(F_m\geq0\) and \(F_m-F_{m-1}\geq0\), we can apply Lemma~\ref{lem_4_2} with \(M=N+j\).
Therefore, for \(n\geq N+j+1\), we obtain
\begin{align}
	F_n
	&\geq
	BA_j^p\eps^{p^{j+1}}
	\sum_{m=N+j}^{n-1}
	(n-m)m^{-Q-p\alpha_j}
	(m-N-j)^{p\beta_j}
	\notag\\
	&\geq
	BA_j^p\eps^{p^{j+1}}
	n^{-Q-p\alpha_j}
	\sum_{m=N+j}^{n-1}
	(n-m)(m-N-j)^{p\beta_j}.
	\label{eq_32}
\end{align}
where we have used \(m\leq n\) and \(Q+p\alpha_j\geq0\).

Next, we shall estimate the sum in \eqref{eq_32}.
Let \(b>0\) and \(r\in\N\).
Changing the order of summation and comparing the sums with integrals, we obtain
\begin{align}
	\sum_{k=0}^{r-1}(r-k)k^b
	=
	\sum_{\ell=0}^{r-1}
	\sum_{k=0}^{\ell}k^b
	\geq
	\frac{1}{b+1}
	\sum_{\ell=0}^{r-1}\ell^{b+1}
	\geq
	\frac{1}{b+1}
	\int_0^{r-1}x^{b+1}\,dx
	=
	\frac{(r-1)^{b+2}}
	{(b+1)(b+2)}.
	\label{eq_33}
\end{align}
We note that the first equality holds because \(k^b\) is counted \(r-k\) times.

Applying \eqref{eq_33} to \eqref{eq_32} with
\[
	r : =n-N-j,
	\qquad
	k := m-N-j,
	\qquad
	b := p\beta_j,
\]
we obtain
\[
	F_n
	\geq
	\frac{BA_j^p}
	{(p\beta_j+1)(p\beta_j+2)}
	\eps^{p^{j+1}}
	n^{-Q-p\alpha_j}
	(n-N-j-1)^{p\beta_j+2}.
\]
By the definitions in \eqref{eq_28}--\eqref{eq_30}, we get
\[
	F_n
	\geq
	A_{j+1}\eps^{p^{j+1}}
	n^{-\alpha_{j+1}}
	(n-N-(j+1))^{\beta_{j+1}},
\]
which is the desired estimate at the next step. Hence the induction is complete.

We next solve the recurrence relations for \(\alpha_j\) and \(\beta_j\).
It follows from \eqref{eq_28} and \eqref{eq_29} that
\begin{equation}
	\alpha_j = \frac{Q}{p-1}(p^j-1),
	\qquad
	\beta_j = \frac{p+1}{p-1}p^j - \frac{2}{p-1}.
	\label{eq_34}
\end{equation}
Therefore we have
\begin{equation}
	\beta_j-\alpha_j
	= \frac{\kappa}{p-1}p^j + \frac{Q-2}{p-1}.
	\label{eq_35}
\end{equation}
We note that the positivity of \(\kappa\) makes the power of \(n\) increase at each step of the iteration.

Next, we shall show a lower bound of \(A_j\).
From \eqref{eq_34}, there exists a positive constant \(C_0>0\) such that
\[
	\beta_j \leq C_0p^j,
	\qquad
	j \geq 0.
\]
Making use of this estimate in \eqref{eq_30}, we obtain, with a positive constant \(C_1>0\),
\[
	A_{j+1} \geq C_1p^{-2j}A_j^p,
	\qquad
	j \geq 0.
\]
Taking the logarithm of this inequality, we obtain
\[
	\log A_{j+1}
	\geq
	p\log A_j+\log C_1-2j\log p.
\]
Dividing both sides by \(p^{j+1}\), we get
\[
	\frac{\log A_{j+1}}{p^{j+1}}
	\geq \frac{\log A_j}{p^j}
	+ \frac{\log C_1-2j\log p}{p^{j+1}}.
\]
Summing these inequalities from \(0\) to \(j-1\), we obtain
\begin{equation}
	\frac{\log A_j}{p^j}
	\geq
	\log A
	+
	\sum_{k=0}^{j-1}
	\frac{\log C_1-2k\log p}{p^{k+1}}.
	\label{eq_36}
\end{equation}
Since
\[
	\sum_{k=0}^{\infty}
	\frac{1+k}{p^{k+1}}
	<\infty,
\]
we see that the right hand side of \eqref{eq_36} is bounded from below independently of \(j\).
Therefore there exists a positive constant \(D>0\) such that
\begin{equation}
	A_j\geq D^{p^j},
	\qquad
	j\geq0.
	\label{eq_37}
\end{equation}

Let $2N+4\leq n<L$. We put
\[
	j := \left\lfloor\frac{n-N}{2} \right\rfloor.
\]
Then \(N+j\leq n\) holds, and we have
\[
	n-N-j
	\geq \frac{n-N}{2}
	\geq \frac{n}{4},
\]
because of \(n\geq2N\). Similarly, we have
\[
	j
	\geq \frac{n-N}{2}-1
	\geq \frac{n}{4},
\]
because of \(n\geq2N+4\).

Making use of \eqref{eq_31}, \eqref{eq_34}, \eqref{eq_35}, and \eqref{eq_37}, we obtain
\[
	F_n
	\geq
	D^{p^j}\eps^{p^j}
	n^{-\alpha_j}
	\left(\frac{n}{4}\right)^{\beta_j}
	=
	4^{\frac{2}{p-1}}
	n^{\frac{Q-2}{p-1}}
	\left\{
		D4^{-\frac{p+1}{p-1}}
		\eps
		n^{\frac{\kappa}{p-1}}
	\right\}^{p^j}.
\]
We put
\[
	C_2:=4^{\frac{2}{p-1}},
	\qquad
	D_1:=D4^{-\frac{p+1}{p-1}}.
\]
Then the above estimate can be written as
\begin{equation}
	F_n
	\geq
	C_2
	n^{\frac{Q-2}{p-1}}
	\left(
		D_1\eps
		n^{\frac{\kappa}{p-1}}
	\right)^{p^j}.
	\label{eq_38}
\end{equation}

We next compare the lower bound with the upper bound involving \(\Lambda\). We put
\[
	\ell_\Lambda := 1+\log\log(e+\Lambda).
\]
We shall show that there exists a positive constant \(C_3>0\), independent of \(\Lambda\), such that
\begin{equation}
	C_2 n^{\frac{Q-2}{p-1}} 2^{p^{n/4}}
	> K\Lambda n^s,
	\qquad
	n\geq C_3\ell_\Lambda.
	\label{eq_39}
\end{equation}
Taking logarithms, we see that it is sufficient to show that
\begin{equation}
	(\log2)p^{n/4}
	> \log\Lambda
	+ \left(s-\frac{Q-2}{p-1}\right)\log n
	+ \log\frac{K}{C_2}.
	\label{eq_40}
\end{equation}

We put
\[
	a := \frac{\log p}{4},
	\qquad
	\alpha := s-\frac{Q-2}{p-1},
\]
and
\[
	\alpha_+:=\max\{\alpha,0\},
	\qquad
	b_+ := \max\left\{\log\frac{K}{C_2},0\right\}.
\]
If \(n\geq C_3\ell_\Lambda\), then we have
\[
	\log\log(e+\Lambda)
	\leq
	\frac{n}{C_3},
\]
and hence
\[
	\log(e+\Lambda)
	\leq
	\exp\left(\frac{n}{C_3}\right).
\]
We take \(C_3\) so large that
\[
	\frac{1}{C_3} \leq \frac{a}{2}.
\]
Then we obtain
\[
	\log\Lambda
	\leq \log(e+\Lambda)
	\leq e^{an/2}.
\]
On the other hand, we note that
\[
	xe^{-ax/2} \leq \frac{2}{ae}
	\qquad
	(x > 0),
\]
which implies
\[
	\log n
	\leq n
	\leq \frac{2}{ae}e^{an/2}.
\]
From these estimates, we get
\[
	\log\Lambda
	+ \alpha\log n
	+ \log\frac{K}{C_2}
	\leq \left(1+\frac{2\alpha_+}{ae}+b_+\right)e^{an/2}.
\]
Since \(\ell_\Lambda\geq1\), we can take \(C_3\) larger if necessary so that
\[
	\left(1+\frac{2\alpha_+}{ae}+b_+\right)e^{an/2}
	< (\log2)e^{an}
\]
for \(n\geq C_3\ell_\Lambda\).
Since
\[
	e^{an} = p^{n/4},
\]
we have \eqref{eq_40} for every \(\Lambda>0\).
The constant used here is independent of the threshold.

Now we choose \(C_4>0\) so large that
\begin{equation}
	D_1C_4^{\frac{\kappa}{p-1}} \geq 2,
	\qquad
	C_4 \geq 2N+4,
	\qquad
	C_4 \geq C_3.
	\label{eq_41}
\end{equation}
We put
\begin{equation}
	n_{\eps,\Lambda}
	:=
	\left\lfloor
		C_4
		\left(
			\eps^{-\frac{p-1}{\kappa}}
			+
			\ell_\Lambda
		\right)
	\right\rfloor
	+1.
	\label{eq_42}
\end{equation}
Suppose, on the contrary, that
\[
	n_{\eps,\Lambda} < L.
\]
Since
\[
	n_{\eps,\Lambda} > C_4\eps^{-\frac{p-1}{\kappa}},
\]
we have
\begin{align}
	D_1\eps n_{\eps,\Lambda}^{\frac{\kappa}{p-1}}
	> D_1\eps
	\left(C_4\eps^{-\frac{p-1}{\kappa}}\right)^{\frac{\kappa}{p-1}}
	= D_1C_4^{\frac{\kappa}{p-1}}
	\geq 2.
	\label{eq_43}
\end{align}
It follows from \eqref{eq_41} that
\[
	n_{\eps,\Lambda} \geq 2N+4.
\]
Hence, for
\[
	j_{\eps,\Lambda} := \left\lfloor \frac{n_{\eps,\Lambda}-N}{2} \right\rfloor,
\]
we have
\[
	j_{\eps,\Lambda}
	\geq
	\frac{n_{\eps,\Lambda}}{4}.
\]
From \eqref{eq_38} and \eqref{eq_43}, we obtain
\[
	F_{n_{\eps,\Lambda}}
	\geq C_2 n_{\eps,\Lambda}^{\frac{Q-2}{p-1}} 2^{p^{j_{\eps,\Lambda}}}
	\geq C_2 n_{\eps,\Lambda}^{\frac{Q-2}{p-1}} 2^{p^{n_{\eps,\Lambda}/4}}.
\]
On the other hand,
\[
	n_{\eps,\Lambda}
	> C_4\ell_\Lambda
	\geq C_3\ell_\Lambda.
\]
It follows from \eqref{eq_39} that
\[
	F_{n_{\eps,\Lambda}}
	>
	K\Lambda n_{\eps,\Lambda}^s,
\]
which contradicts \eqref{eq_24}. Therefore we conclude that
\[
	L \leq n_{\eps,\Lambda}.
\]
By \eqref{eq_42}, we have
\[
	L \leq C \left\{\eps^{-\frac{p-1}{\kappa}} + 1+\log\log(e+\Lambda)\right\}
\]
with a positive constant \(C>0\) independent of \(\eps\) and \(\Lambda\).
This is \eqref{eq_27}, and the proof is complete.
\end{proof}

\subsection{Critical logarithmic iteration}
\label{subsec_5_2}

When \(\kappa=0\), we have \(Q=p+1\).
In this case, the inequality in Proposition~\ref{prop_4_3} becomes
\[
	U_{n+1}-2U_n+U_{n-1}
	\geq
	Bn^{-(p+1)}U_n^p.
\]
Here the positive power of \(n\) obtained by the subcritical iteration vanishes.
We shall instead carry out the iteration with logarithmic factors.

\begin{lemma}[Critical logarithmic iteration]
	\label{lem_5_2}
Let \(p>1\), \(N\in\N\), and \(L\in\N\cup\{\infty\}\) satisfy \(N<L\).
Let \((F_n)\) be a nonnegative sequence defined for $N-1\leq n\leq L$ when \(L<\infty\),
and for \(n\geq N-1\) when \(L=\infty\).
Suppose that there exist constants $A,B,K>0$, $s>0$, $0<\eps\leq1$, and $\Lambda>0$ such that
\begin{align}
	&F_n-F_{n-1}\geq0,
	\qquad
	F_n\leq K\Lambda n^s,
	\label{eq_44}\\
	&F_n\geq A\eps n,
	\label{eq_45}\\
	&F_{n+1}-2F_n+F_{n-1} \geq Bn^{-(p+1)}F_n^p
	\label{eq_46}
\end{align}
for \(N\leq n<L\).
Then there exist positive constants \(C_1,C_2,C_3>0\),
independent of \(\eps\in(0,1]\) and \(\Lambda>0\), such that
\begin{equation}
	L
	\leq C_1 \exp\left(C_2\eps^{-(p-1)}\right)
	+ C_3 \left\{1+\log(e+\Lambda)\right\}^{\frac{2}{\log_4p}}.
	\label{eq_47}
\end{equation}
In particular, there exists a positive constant \(C>0\), independent of \(\eps\) and \(\Lambda\), such that
\begin{equation}
	L \leq C \exp\left(C\eps^{-(p-1)}\right)
	\left\{1+\log(e+\Lambda)\right\}^{\frac{2}{\log_4p}}.
	\label{eq_48}
\end{equation}
\end{lemma}

\begin{proof}
We put, for \(j\geq0\),
\[
	L_j := 4^jN,
	\qquad
	b_0 := 0,
	\qquad
	b_{j+1} := pb_j+1,
	\qquad
	A_0 := A.
\]
It is easy to solve the recurrence relation for \(b_j\) as
\begin{equation}
	b_j = \frac{p^j-1}{p-1}.
	\label{eq_49}
\end{equation}

We are going to prove that, for all \(j\geq0\),
\begin{equation}
	F_n
	\geq
	A_j\eps^{p^j}n
	\left(
		\log\frac{n}{L_j}
	\right)^{b_j},
	\qquad
	L_j\leq n<L,
	\label{eq_50}
\end{equation}
where we understand the logarithmic factor for \(j=0\) to be \(1\).

It is apparent that, for \(j=0\), \eqref{eq_50} holds by \eqref{eq_45}.
Suppose that \eqref{eq_50} is true for some \(j\geq0\). Let
\[
	L_{j+1}=4L_j \leq n<L.
\]
For
\[
	L_j+1 \leq m \leq\left\lfloor\frac{n}{2}\right\rfloor,
\]
we obtain from the induction hypothesis and \eqref{eq_46} that
\[
	F_{m+1}-2F_m+F_{m-1}
	\geq
	Bm^{-(p+1)}
	\left[
		A_j\eps^{p^j}m
		\left(
			\log\frac{m}{L_j}
		\right)^{b_j}
	\right]^p
	=
	BA_j^p\eps^{p^{j+1}}
	\frac{1}{m}
	\left(
		\log\frac{m}{L_j}
	\right)^{pb_j}.
\]

Applying Lemma~\ref{lem_4_2} with \(M=L_j\), we obtain a sum with nonnegative summands.
We may therefore restrict the sum to
\[
	L_j+1\leq m\leq\left\lfloor\frac{n}{2}\right\rfloor,
\]
and obtain
\begin{equation}
	F_n
	\geq B \sum_{m=L_j+1}^{\lfloor n/2\rfloor} (n-m)m^{-(p+1)}F_m^p
	\geq
	\frac{BA_j^p}{2} \eps^{p^{j+1}}n
	\sum_{m=L_j+1}^{\lfloor n/2\rfloor}
	\frac{1}{m} \left(\log\frac{m}{L_j}\right)^{pb_j}.
	\label{eq_51}
\end{equation}
where we have used \(m\leq n/2\) and \(n-m\geq n/2\).

Next, we shall estimate the logarithmic sum.
For \(r\geq0\), we put
\[
	\phi(x) := \left(\log\frac{x}{L_j}\right)^{r+1},
	\qquad
	x \geq L_j.
\]
For \(m\geq L_j+1\), the mean value theorem gives a point \(\xi_m\in(m-1,m)\) such that
\[
	\phi(m)-\phi(m-1)
	= \frac{r+1}{\xi_m} \left(\log\frac{\xi_m}{L_j}\right)^r.
\]
Since \(N\geq1\), we have \(L_j\geq1\), which implies \(m\geq L_j+1\geq2\).
Hence we have
\[
	\xi_m > m-1 \geq \frac{m}{2},
	\qquad
	\log\frac{\xi_m}{L_j} \leq \log\frac{m}{L_j}.
\]
From these inequalities, we obtain
\[
	\phi(m)-\phi(m-1)
	\leq \frac{2(r+1)}{m}\left(\log\frac{m}{L_j}\right)^r,
\]
that is,
\begin{equation}
	\frac{1}{m}
	\left(
		\log\frac{m}{L_j}
	\right)^r
	\geq
	\frac{1}{2(r+1)}
	\left[
		\left(
			\log\frac{m}{L_j}
		\right)^{r+1}
		-
		\left(
			\log\frac{m-1}{L_j}
		\right)^{r+1}
	\right].
	\label{eq_52}
\end{equation}

Summing \eqref{eq_52} from \(m=L_j+1\) to \(m=\lfloor n/2\rfloor\), we have
\[
	\sum_{m=L_j+1}^{\lfloor n/2\rfloor}
	\frac{1}{m}
	\left(
		\log\frac{m}{L_j}
	\right)^r
	\geq
	\frac{1}{2(r+1)}
	\left(
		\log
		\frac{\lfloor n/2\rfloor}{L_j}
	\right)^{r+1}.
\]
Since $n\geq L_{j+1}=4L_j\geq4$, we have $\left\lfloor n/2\right\rfloor \geq n/4$.
Hence
\[
	\frac{\lfloor n/2\rfloor}{L_j}
	\geq \frac{n}{4L_j}
	= \frac{n}{L_{j+1}}.
\]
It follows that
\begin{equation}
	\sum_{m=L_j+1}^{\lfloor n/2\rfloor}
	\frac{1}{m}
	\left(
		\log\frac{m}{L_j}
	\right)^r
	\geq
	\frac{1}{2(r+1)}
	\left(
		\log\frac{n}{L_{j+1}}
	\right)^{r+1}.
	\label{eq_53}
\end{equation}

Substituting \eqref{eq_53} into \eqref{eq_51} with \(r=pb_j\), we obtain
\[
	F_n
	\geq
	\frac{BA_j^p}{4(pb_j+1)}
	\eps^{p^{j+1}}n
	\left(\log\frac{n}{L_{j+1}}\right)^{pb_j+1}.
\]
This gives the desired estimate at the next step if we define
\begin{equation}
	A_{j+1} := \frac{B}{4(pb_j+1)}A_j^p.
	\label{eq_54}
\end{equation}

The induction is therefore complete.
We next estimate \(A_j\) from below.
From \eqref{eq_49}, we have
\[
	pb_j+1
	= \frac{p^{j+1}-1}{p-1}
	\leq \frac{p}{p-1}p^j.
\]
Making use of this estimate in \eqref{eq_54}, we obtain a positive constant \(C_A>0\) such that
\[
	A_{j+1} \geq C_Ap^{-j}A_j^p,
	\qquad
	j \geq 0.
\]
Taking logarithms and dividing by \(p^{j+1}\), we can repeat the resulting inequality to obtain
\begin{equation}
	\frac{\log A_j}{p^j}
	\geq \log A + \sum_{k=0}^{j-1}\frac{\log C_A-k\log p}{p^{k+1}}.
	\label{eq_55}
\end{equation}
Since
\[
	\sum_{k=0}^{\infty} \frac{1+k}{p^{k+1}} < \infty,
\]
the right hand side of \eqref{eq_55} is bounded from below independently of \(j\).
Thus there exists a positive constant \(D>0\) such that
\begin{equation}
	A_j \geq D^{p^j},
	\qquad
	j \geq 0.
	\label{eq_56}
\end{equation}

We shall compare the lower bound with the upper bound given by the threshold condition. We put
\[
	m_\Lambda := 1+\log(e+\Lambda),
	\qquad
	\gamma := \frac{1}{2}\log_4p.
\]
We note that \(p>1\) implies \(\gamma>0\).
We shall show that there exists a positive constant \(C_3>0\), independent of \(\Lambda\), such that
\begin{equation}
	n \left(\log\frac{n}{N}\right)^{-\frac{1}{p-1}}
	2^{p^{-1}(n/N)^\gamma}
	> K\Lambda n^s
	\label{eq_57}
\end{equation}
provided that
\begin{equation}
	n \geq C_3m_\Lambda^{1/\gamma}.
	\label{eq_58}
\end{equation}

We put
\[
	a := \frac{\log2}{pN^\gamma}
\]
and
\[
	B_0 := (s-1)_+ + \frac{1}{p-1},
	\qquad
	(s-1)_+ := \max\{s-1,0\}.
\]
Taking logarithms, we see that it is sufficient to prove
\begin{equation}
	an^\gamma
	> \log K + \log\Lambda +
	(s-1)\log n + \frac{1}{p-1} \log\log\frac{n}{N}.
	\label{eq_59}
\end{equation}

Since \(\log\Lambda\leq m_\Lambda\) and
\[
	\log\log\frac{n}{N}
	\leq
	\log n
	\qquad
	(n\geq eN),
\]
we get
\begin{align}
	\log K
	+ \log\Lambda
	+ (s-1)\log n
	+ \frac{1}{p-1} \log\log\frac{n}{N}
	\leq (\log K)_+
	+ m_\Lambda
	+ B_0\log n,
	\label{eq_60}
\end{align}
where
\[
	(\log K)_+ := \max\{\log K,0\}.
\]

We shall use the elementary inequality
\begin{equation}
	\log x
	\leq
	\frac{2}{\gamma e}x^{\gamma/2},
	\qquad
	x\geq1.
	\label{eq_61}
\end{equation}
We choose \(M\geq eN\) so large that
\begin{equation}
	M^{\gamma/2} \geq \frac{8B_0}{a\gamma e}
	\label{eq_62}
\end{equation}
and
\begin{equation}
	M^\gamma
	\geq
	\frac{4(\log K)_+}{a}.
	\label{eq_63}
\end{equation}
For \(n\geq M\), it follows from \eqref{eq_61} and \eqref{eq_62} that
\begin{align}
	B_0\log n
	\leq \frac{2B_0}{\gamma e}n^{\gamma/2}
	= \frac{2B_0}{\gamma e} n^\gamma n^{-\gamma/2}
	\leq \frac{2B_0}{\gamma e} M^{-\gamma/2}n^\gamma
	\leq \frac{a}{4}n^\gamma.
	\label{eq_64}
\end{align}
From \eqref{eq_63}, we also have
\begin{equation}
	(\log K)_+ \leq \frac{a}{4}n^\gamma.
	\label{eq_65}
\end{equation}

We next choose \(C_3\geq M\) so large that
\[
	C_3^\gamma \geq \frac{4}{a}.
\]
Assume that \eqref{eq_58} holds.
From \(m_\Lambda\geq1\), we have \(n\geq C_3\geq M\), and also
\[
	n^\gamma
	\geq
	C_3^\gamma m_\Lambda.
\]
Thus we obtain
\begin{equation}
	m_\Lambda
	\leq
	C_3^{-\gamma}n^\gamma
	\leq
	\frac{a}{4}n^\gamma.
	\label{eq_66}
\end{equation}
Making use of \eqref{eq_60}, \eqref{eq_64}, \eqref{eq_65}, and \eqref{eq_66}, we obtain
\[
	\log K
	+ \log\Lambda
	+ (s-1)\log n
	+ \frac{1}{p-1}\log\log\frac{n}{N}
	\leq \frac{3a}{4}n^\gamma
	< an^\gamma.
\]
This is \eqref{eq_59}. Hence \eqref{eq_57} holds for every \(\Lambda>0\), with a constant independent of the threshold.

Now we take \(C_4>0\) so large that
\begin{equation}
	D \left(\frac{C_4}{2}\right)^{\frac{1}{p-1}} \geq 2,
	\qquad
	e^{C_4} \geq 4.
	\label{eq_67}
\end{equation}
We put
\begin{equation}
	n_{\eps,\Lambda}
	:=
	\left\lfloor
		N \exp\left(C_4\eps^{-(p-1)}\right)
		+ C_3m_\Lambda^{1/\gamma}
	\right\rfloor
	+ 1.
	\label{eq_68}
\end{equation}
We also define
\[
	J
	:=
	\left\lfloor
		\frac{1}{2}
		\log_4
		\frac{n_{\eps,\Lambda}}{N}
	\right\rfloor.
\]

Making use of \(0<\eps\leq1\) and \eqref{eq_67}, we have
\[
	n_{\eps,\Lambda}
	> N \exp\left(C_4\eps^{-(p-1)}\right)
	\geq Ne^{C_4}
	\geq 4N.
\]
Thus \(J\geq0\). We also have
\[
	4^J \leq \left(\frac{n_{\eps,\Lambda}}{N}\right)^{1/2},
\]
which implies
\[
	L_J
	= 4^JN
	\leq \sqrt{n_{\eps,\Lambda}N}
	\leq n_{\eps,\Lambda}.
\]
From this, we obtain
\begin{align}
	\log \frac{n_{\eps,\Lambda}}{L_J}
	\geq \frac{1}{2} \log \frac{n_{\eps,\Lambda}}{N}
	> \frac{C_4}{2} \eps^{-(p-1)}.
	\label{eq_69}
\end{align}

Suppose, on the contrary, that
\[
	n_{\eps,\Lambda}<L.
\]
We can then use \eqref{eq_50} with \(j=J\).
Making use of \eqref{eq_49} and \eqref{eq_56}, we obtain
\[
	F_{n_{\eps,\Lambda}}
	\geq
	n_{\eps,\Lambda}
	\left(\log\frac{n_{\eps,\Lambda}}{L_J}\right)^{-\frac{1}{p-1}}
	\left\{
		D\eps \left(\log\frac{n_{\eps,\Lambda}}{L_J}\right)^{\frac{1}{p-1}}
	\right\}^{p^J}.
\]
By \eqref{eq_69} and \eqref{eq_67}, we have
\[
	D\eps \left(\log\frac{n_{\eps,\Lambda}}{L_J}\right)^{\frac{1}{p-1}}
	> D\eps \left(\frac{C_4}{2}\eps^{-(p-1)}\right)^{\frac{1}{p-1}}
	= D \left(\frac{C_4}{2}\right)^{\frac{1}{p-1}}
	\geq 2.
\]
Hence we get
\begin{equation}
	F_{n_{\eps,\Lambda}}
	\geq n_{\eps,\Lambda}
	\left(\log\frac{n_{\eps,\Lambda}}{L_J}\right)^{-\frac{1}{p-1}} 2^{p^J}.
	\label{eq_70}
\end{equation}

Next, we shall estimate \(p^J\). Since
\[
	J \geq \frac{1}{2} \log_4 \frac{n_{\eps,\Lambda}}{N} - 1,
\]
we get
\begin{align}
	p^J
	\geq p^{-1} p^{\frac{1}{2}\log_4(n_{\eps,\Lambda}/N)}
	= p^{-1} \left(\frac{n_{\eps,\Lambda}}{N}\right)^{\frac{1}{2}\log_4p}
	= p^{-1}\left(\frac{n_{\eps,\Lambda}}{N}\right)^\gamma.
	\label{eq_71}
\end{align}
On the other hand, noticing that \(L_J\geq N\) and \(-1/(p-1)<0\), we have
\[
	\left(\log\frac{n_{\eps,\Lambda}}{L_J}\right)^{-\frac{1}{p-1}}
	\geq
	\left(\log\frac{n_{\eps,\Lambda}}{N}\right)^{-\frac{1}{p-1}}.
\]
Making use of this inequality, \eqref{eq_70} and \eqref{eq_71}, we obtain
\[
	F_{n_{\eps,\Lambda}}
	\geq
	n_{\eps,\Lambda}
	\left(\log\frac{n_{\eps,\Lambda}}{N}\right)^{-\frac{1}{p-1}}
	2^{p^{-1}(n_{\eps,\Lambda}/N)^\gamma}.
\]

By the definition of \(n_{\eps,\Lambda}\), we have
\[
	n_{\eps,\Lambda} > C_3m_\Lambda^{1/\gamma}.
\]
Therefore \eqref{eq_57} gives us
\[
	F_{n_{\eps,\Lambda}} > K\Lambda n_{\eps,\Lambda}^s,
\]
which is a contradiction to \eqref{eq_44}. Hence we conclude that
\[
	L\leq n_{\eps,\Lambda}.
\]

It follows from \eqref{eq_68} that
\begin{align*}
	L
	&\leq
	N \exp\left(C_4\eps^{-(p-1)}\right)
	+ C_3 \left\{1+\log(e+\Lambda)\right\}^{1/\gamma} +1 \\
	&\leq
	\widetilde C_1 \exp\left(\widetilde C_2\eps^{-(p-1)}\right)
	+ \widetilde C_3 \left\{1+\log(e+\Lambda)\right\}^{1/\gamma}
\end{align*}
for some positive constants \(\widetilde C_1,\widetilde C_2,\widetilde C_3>0\)
independent of \(\eps\) and \(\Lambda\).
Noticing that
\[
	\frac{1}{\gamma} = \frac{2}{\log_4p},
\]
we obtain \eqref{eq_47}.

Finally, we have
\[
	\exp\left(\widetilde C_2\eps^{-(p-1)}\right) \geq 1
\]
and
\[
	\left\{1+\log(e+\Lambda)\right\}^{1/\gamma} \geq 1.
\]
From these inequalities, we obtain
\[
	\exp\left(\widetilde C_2\eps^{-(p-1)}\right)
	+ \left\{1+\log(e+\Lambda)\right\}^{1/\gamma}
	\leq
	2 \exp\left(\widetilde C_2\eps^{-(p-1)}\right)
	\left\{1+\log(e+\Lambda)\right\}^{1/\gamma}.
\]
Taking the constant in front of the exponential sufficiently large, we get
\[
	L
	\leq
	C \exp\left(C\eps^{-(p-1)}\right)
	\left\{1+\log(e+\Lambda)\right\}^{\frac{2}{\log_4p}},
\]
where \(C>0\) is independent of \(\eps\) and \(\Lambda\).
This is \eqref{eq_48}. The proof is complete.
\end{proof}

\section{Proofs of the main results}
\label{sec_6}

\begin{proof}[Proof of Theorem~\ref{thm_3_1}]
Fix \(\Lambda>0\). We put
\[
	L := N_\Lambda(u^\eps).
\]
We first consider \(L\leq1\).
Since \(0<\eps\leq1\), the right hand sides of the estimates in Theorem~\ref{thm_3_1} have a positive lower bound independent of \(\eps\) and \(\Lambda\).
Hence the desired estimates hold by taking the constants sufficiently large.
In what follows, we assume $L>1$ and put
\[
	U_n := \sum_{i\in\Z^d}u_n^{\eps,(i)},
	\qquad
	n \in \N_0.
\]
By Proposition~\ref{prop_4_3}, we can take positive constants \(A,B,K>0\), independent of \(\eps\) and \(\Lambda\), such that the following estimates hold for \(1\leq n<L\):
\begin{align}
	&U_n-U_{n-1} > 0,
	\qquad
	0 \leq U_n \leq K\Lambda n^d,
	\notag  \\
	&U_n \geq A\eps n,
	\label{eq_72}
	\\
	&U_{n+1}-2U_n+U_{n-1} \geq Bn^{-Q}U_n^p,
	\qquad
	Q = d(p-1)+q.
	\label{eq_73}
\end{align}

First, we shall show the estimate for the subcritical case
\[
	\kappa := p+1-Q > 0.
\]
Noticing that \(n\geq1\), we obtain from \eqref{eq_72} that
\[
	U_n
	\geq A\eps n
	\geq A\eps(n-1).
\]
Thus all the assumptions of Lemma~\ref{lem_5_1} are satisfied with
\[
	F_n := U_n,
	\qquad
	N := 1,
	\qquad
	s := d.
\]
Applying this lemma, we obtain
\[
	L \leq C \left\{\eps^{-\frac{p-1}{\kappa}} + 1+\log\log(e+\Lambda)\right\},
\]
where \(C>0\) is independent of \(\eps\) and \(\Lambda\).
Since $L=N_\Lambda(u^\eps)$, this gives us
\[
	N_\Lambda(u^\eps)
	\leq
	C \left\{\eps^{-\frac{p-1}{\kappa}} + 1+\log\log(e+\Lambda)\right\}.
\]
This proves the estimate in the subcritical case.

Next, we shall show the estimate for the critical case
\[
	\kappa = 0.
\]
By the definition of \(\kappa\), we have $Q=p+1$.
Therefore \eqref{eq_73} is written as
\[
	U_{n+1}-2U_n+U_{n-1}
	\geq
	Bn^{-(p+1)}U_n^p.
\]
All the assumptions of Lemma~\ref{lem_5_2} are now satisfied with
\[
	F_n := U_n,
	\qquad
	N := 1,
	\qquad
	s := d.
\]
It follows from the first estimate in Lemma~\ref{lem_5_2} that
\[
	L
	\leq
	C_1 \exp\left(C_2\eps^{-(p-1)}\right)
	+ C_3 \left\{1+\log(e+\Lambda)\right\}^{\frac{2}{\log_4p}},
\]
where \(C_1,C_2,C_3>0\) are independent of \(\eps\) and \(\Lambda\).
Therefore we have
\[
	N_\Lambda(u^\eps)
	\leq C_1 \exp\left(C_2\eps^{-(p-1)}\right)
	+ C_3 \left\{1+\log(e+\Lambda)\right\}^{\frac{2}{\log_4p}}.
\]
The multiplicative estimate follows from the second estimate in Lemma~\ref{lem_5_2}.
This completes the proof.
\end{proof}

\begin{proof}[Proof of Corollary~\ref{cor_undamped}]
When \(q=0\), we have $Q=d(p-1)$ and
\[
	\kappa = p+1-Q
	= p+1-d(p-1)
	= 2-(d-1)(p-1).
\]

If \(d=1\), then $\kappa=2>0$ for all \(p>1\).
Thus the subcritical estimate in Theorem~\ref{thm_3_1} gives us
\[
	N_\Lambda(u^\eps)
	\leq
	C \left\{\eps^{-\frac{p-1}{2}}+ 1+\log\log(e+\Lambda)\right\}.
\]

Next, assume that \(d\geq2\).
The condition \(\kappa>0\) is equivalent to
\[
	2-(d-1)(p-1)>0,
	\qquad
	\text{or}
	\qquad
	1<p<\frac{d+1}{d-1}.
\]
Applying Theorem~\ref{thm_3_1}, we obtain
\[
	N_\Lambda(u^\eps)
	\leq
	C
	\left\{
		\eps^{-\frac{p-1}{2-(d-1)(p-1)}} + 1+\log\log(e+\Lambda)
	\right\}.
\]

Finally, \(\kappa=0\) is equivalent to $p= (d+1)/(d-1)$.
Making use of the critical estimate in Theorem~\ref{thm_3_1}, we obtain
\[
	N_\Lambda(u^\eps)
	\leq C_1 \exp\left(C_2\eps^{-(p-1)}\right)
	+ C_3 \left\{1+\log(e+\Lambda)\right\}^{\frac{2}{\log_4p}}.
\]
This completes the proof.
\end{proof}

\begin{proof}[Proof of Corollary~\ref{cor_scale_invariant}]
For $q=p-1$, we have
\[
	Q = d(p-1)+q
	= (d+1)(p-1),
\]
and hence
\[
	\kappa
	= p+1-Q
	= 2-d(p-1).
\]

We note that \(\kappa>0\) is equivalent to $1<p<1+ 2/d$.
Hence the subcritical estimate in Theorem~\ref{thm_3_1} gives us
\[
	N_\Lambda(u^\eps)
	\leq
	C \left\{\eps^{-\frac{p-1}{2-d(p-1)}} + 1+\log\log(e+\Lambda)\right\}.
\]

For \(\kappa=0\), we have $p=1+2/d$.
Applying the critical estimate in Theorem~\ref{thm_3_1}, we obtain
\[
	N_\Lambda(u^\eps)
	\leq C_1 \exp\left(C_2\eps^{-(p-1)}\right)
	+ C_3 \left\{1+\log(e+\Lambda)\right\}^{\frac{2}{\log_4p}}.
\]
This completes the proof.
\end{proof}

\begin{proof}[Proof of Corollary~\ref{cor_3_3}]
Let \(\Lambda>0\) be arbitrarily fixed.
If $\Theta_0(u)\geq\Lambda$, then $N_\Lambda(u)=0$.
If $\Theta_0(u)<\Lambda$ and $\Theta_1(u)\geq\Lambda$, then $N_\Lambda(u)=1$.
In both cases, we have $N_\Lambda(u)<\infty$.

Hence it suffices to consider the case
\[
	\Theta_0(u)<\Lambda
	\qquad\text{and}\qquad
	\Theta_1(u)<\Lambda.
\]
We put
\[
	f_i := u_0^{(i)},
	\qquad
	g_i := \delta^{-1}\left(u_1^{(i)}-u_0^{(i)}\right),
	\qquad
	i \in \Z^d,
\]
and set
\[
	\eps := 1.
\]
Then we have
\[
	u_0^{(i)} = \eps f_i,
	\qquad
	u_1^{(i)} = \eps\bigl(f_i+\delta g_i\bigr),
\]
so that the initial data have the form \eqref{eq_9}.

It follows from \eqref{eq_13} that
\[
	\supp f\cup\supp g
	\subset
	\left\{i\in\Z^d:\|i\|_1\leq R\right\}.
\]
Indeed, if $u_0^{(i)}=u_1^{(i)}=0$, then $f_i=g_i=0$.
It also follows from \eqref{eq_14} that
\[
	\sum_{i\in\Z^d}f_i
	= \sum_{i\in\Z^d}u_0^{(i)}
	\geq 0
	\qquad
	\text{and}
	\qquad
	\sum_{i\in\Z^d}g_i
	= \delta^{-1} \sum_{i\in\Z^d} \left(u_1^{(i)}-u_0^{(i)}\right)
	> 0.
\]
Hence all the assumptions of Theorem~\ref{thm_3_1} are satisfied.
Since we have assumed
\[
	\kappa
	= p+1-d(p-1)-q
	\geq 0,
\]
we can use either the subcritical or the critical estimate in Theorem~\ref{thm_3_1}.
Taking \(\eps=1\), we obtain
\[
	N_\Lambda(u)<\infty.
\]

As \(\Lambda>0\) was arbitrarily fixed, every finite amplitude threshold is reached in finitely many time steps.
This completes the proof.
\end{proof}

\section{Conclusion}
\label{sec_7}

In this paper, we have considered the standard centered finite-difference discretization of
\[
  u_{tt}-\Delta u=(1+t)^{-q}|u|^p,
  \qquad q\geq0,
\]
as a difference equation.
In contrast to the discrete models with a singular nonlinear term, the solution is defined at every finite time step.
We have therefore considered the threshold-crossing time \(N_\Lambda(u^\eps)\), instead of a finite-step blow-up time, in order to estimate the growth of the solution.

For compactly supported initial data satisfying the conditions on the total displacement and total velocity, we have derived upper bounds of \(N_\Lambda(u^\eps)\).
These bounds give the dependence on both the size \(\eps\) of the initial data and the threshold \(\Lambda\) explicitly.
The proof is divided into two cases according to
\[
  Q=d(p-1)+q,
  \qquad
  \kappa=p+1-Q.
\]
In the case \(\kappa>0\), we have obtained a power upper bound in the size of the initial data by a discrete Kato-type iteration.
In the case \(\kappa=0\), we have obtained an exponential upper bound by a logarithmic iteration.
Namely, for any fixed \(\Lambda>0\) and sufficiently small \(\eps\), we have
\[
  N_\Lambda(u^\eps)
  \leq
  C\eps^{-\frac{p-1}{\kappa}}
  \qquad (\kappa>0),
\]
and
\[
  N_\Lambda(u^\eps)
  \leq
  \exp\!\left(
    C\eps^{-(p-1)}
  \right)
  \qquad (\kappa=0),
\]
where \(C\) is a positive constant independent of \(\eps\) and \(\Lambda\).
For a fixed threshold, the orders in the size of the initial data are the same as those given by the Kato-type argument for the corresponding continuous problem.
We have also proved that, if \(\kappa\geq0\), every finite amplitude threshold is reached in finitely many time steps.
In particular, the solution is unbounded in time.

As applications, we have considered the following two equations.
For the undamped equation \(q=0\), we have
\[
  \kappa=2-(d-1)(p-1),
\]
while for the Liouville-transformed scale-invariantly damped equation \(q=p-1\), we have
\[
  \kappa=2-d(p-1).
\]
Thus the scalar difference inequality shows how the decay of the nonlinear coefficient changes the range in which the discrete Kato-type argument can be applied.

We again note the difference between our problem and the discrete blow-up results of Matsuya and Tsubota et al.
Their nonlinear term has a singular level depending on the mesh sizes, and finite-step blow-up is related to this level.
There is no such singular level in our equation.
The threshold-crossing time is defined directly for the standard centered discretization with a power-type nonlinear term.
Nevertheless, for a fixed threshold, we obtain the same small-data scale as in the continuous Kato-type argument.
This suggests that the scale is due to the growth of the discrete solution, not to a singularity of the recurrence.

We close this paper by mentioning some further problems.
First, we have proved only upper bounds of the threshold-crossing time.
To show that the order in the size of the initial data is optimal, it is necessary to prove corresponding lower bounds.
For a fixed threshold \(\Lambda\), we would like to know whether
\[
  N_\Lambda(u^\eps)
  \asymp
  \eps^{-\frac{p-1}{\kappa}}
\]
holds in the subcritical case, and whether the corresponding exponential estimate holds in the critical case.
Such estimates would clarify to what extent the small-data order of the threshold-crossing time agrees with that obtained by the continuous Kato-type argument.

Second, the case
\[
  \kappa<0
\]
is not treated by our iteration based on the spatial sum.
We should note that \(\kappa=0\) is critical for this argument, and is not asserted to give the sharp critical exponent for the full discrete system.
For the supercritical case, it is interesting to know whether sufficiently small initial data give a uniformly bounded solution, or whether the solution can be unbounded by a mechanism which the spatial sum does not detect.
In higher space dimensions, for the undamped equation, this problem is closely related to extending the result from the Kato-type range toward the Strauss-type range.

Third, we would like to extend the result to more general semilinear difference equations.
In the proof, we have used the finite propagation property and a lower bound of the time-dependent nonlinear coefficient.
Hence we may expect similar threshold-crossing estimates for more general weights \(\rho_n\), nonlinearities, or discretizations of semilinear wave equations.
It is also interesting to consider which parts of the proof apply to other damping terms or to equations with additional lower-order terms.

Finally, \(\delta\) and \(h\) have been fixed throughout the paper.
It remains to obtain estimates uniform with respect to the mesh parameters and to clarify the relation between the discrete threshold-crossing time and the continuous blow-up behavior as the mesh sizes tend to zero.

\section*{Acknowledgments}
The first author is partially supported by JST SPRING (No. JPMJSP2119).
The second author is partially supported by the Grant-in-Aid
for Scientific Research (C) (No. 25K07093),
Japan Society for the Promotion of Science.

The authors are grateful to Professors Hideo Kubo, Satoshi Masaki,
Keisuke Matsuya, Takiko Sasaki, Hiroyuki Takamura,
and Tetsuji Tokihiro for their useful comments.

During the preparation of this manuscript,
the authors used OpenAI ChatGPT-5.6 Sol to improve readability and to carry out preliminary consistency checks.
The author independently reviewed and verified the mathematical content and citations,
made all final revisions,
and takes full responsibility for the manuscript.

\end{document}